\documentclass[11pt,reqno]{amsart}
\usepackage{a4wide}

\usepackage{mathrsfs}
\usepackage{amsfonts}
\usepackage{amsmath}
\usepackage{stmaryrd}
\usepackage{amssymb}
\usepackage{amsthm}
\usepackage{esint}

\usepackage{url}

\usepackage{amscd}
\usepackage{indentfirst}
\usepackage{enumerate}
\usepackage{graphicx}
\usepackage{appendix}

\usepackage[numbers,sort&compress]{natbib}

\usepackage{color}
\usepackage[colorlinks, linkcolor=blue, citecolor=blue]{hyperref}

\emergencystretch=\maxdimen
\newcommand{\pa}{\partial}

\renewcommand{\div}{{\rm div}}
\newcommand{\curl}{{\rm curl}}
\newcommand{\norm}[1]{\left\|#1\right\|}

\newtheorem{thm}{Theorem}[section]
\newtheorem{lem}[thm]{Lemma}

\newtheorem{rem}[thm]{Remark}
\newtheorem{defi}[thm]{Definition}
\numberwithin{equation}{section}

\begin{document}

\title[Blowup criterion for the 3D barotropic compressible Navier-Stokes equations]
{A new blowup criterion for the 3D barotropic compressible Navier-Stokes equations with vacuum}
\author[S. Xu]{Saiguo Xu}
\address[Saiguo Xu]{School of Mathematics and Statistics, Wuhan University, Wuhan, Hubei 430072, P.R. China.}
\email{xsgsxx@126.com}
\author[Y.H. Zhang]{Yinghui Zhang*}
\address[Yinghui Zhang] {School of Mathematics and Statistics, Guangxi Normal University, Guilin, Guangxi 541004, P.R.
China} \email{yinghuizhang@mailbox.gxnu.edu.cn}

\thanks{* Corresponding author.}

\thanks{This work was supported by National Natural Science
Foundation of China $\#$12271114, Guangxi Natural Science Foundation $\#$2024GXNSFDA010071, $\#$2019JJG110003, $\#$2019AC20214, Science and Technology Project of Guangxi $\#$GuikeAD21220114, the Innovation Project of Guangxi Graduate Education $\#$JGY2023061, and the Key Laboratory of Mathematical Model and Application (Guangxi Normal University), Education Department of Guangxi Zhuang Autonomous Region.}

\date{\today}

\begin{abstract}
We investigate the blowup criterion of the barotropic compressible viscous fluids for the Cauchy problem, Dirichlet problem and Navier-slip boundary condition. The main novelty of this paper is two-fold:
First, for the Cauchy problem and Dirichlet problem, we prove that a
strong or smooth solution exists globally, provided that the vorticity of velocity satisfies Serrin's
condition and the maximum norm of the divergence
of the velocity is bounded.
Second, for the Navier-slip boundary condition, we show that if both the maximum norm of the vorticity of velocity and
the maximum norm of the divergence of velocity are bounded, then the solution exists globally.
In particular, this criterion extends the well-known
Beale-Kato-Majda's blowup criterion for the 3D incompressible Euler equations (Comm. Math. Phys. 94(1984):61-66) to the
3D barotropic compressible Navier-Stokes equations, and can be regarded as a complement for the work by Huang-Li-Xin (Comm. Math. Phys. 301(2011):23-35). The vacuum is allowed to exist here.

\end{abstract}

\maketitle

{\small

\section{Introduction}
The motion of a general viscous compressible barotropic fluid in a domain $\Omega\subset\mathbb{R}^3$ is governed by the compressible Navier-Stokes equations
\begin{equation}\label{CNS-eq}
  \begin{cases}
    \pa_t\rho+\div(\rho u)=0, \\
    \pa_t(\rho u)+\div(\rho u\otimes u)-\mu\Delta u-(\mu+\lambda)\nabla\div u+\nabla P(\rho)=0,
  \end{cases}
\end{equation}
where $\rho$, $u$, $P$ denote the density, velocity and pressure respectively, $\mu$ and $\lambda$ represent the shear viscosity and bulk viscosity coefficients which satisfy the physical restrictions:
\begin{equation}
  \mu>0,\quad \lambda+\frac{2}{3}\mu>0.
\end{equation}
The state equation of pressure is given by
\begin{equation*}
  P(\rho)=a\rho^{\gamma}, \quad a>0,\,\gamma>1.
\end{equation*}
The equations \eqref{CNS-eq} will be equipped with initial data
\begin{equation}\label{initial-data}
  (\rho, u)(x,0)=(\rho_0, u_0)(x), \quad x\in\Omega
\end{equation}
and three types of boundary conditions:
\begin{enumerate}[(1)]
  \item Cauchy problem:
  \begin{equation}\label{Cauchy-condition}
    \Omega=\mathbb{R}^3~\textrm{and (in some weak sense)},\, (\rho, u)~\textrm{vannish at }\infty;
  \end{equation}
  \item Dirichlet problem: $\Omega$ is a bounded smooth domain in $\mathbb{R}^3$, and
      \begin{equation}\label{Dirichlet-condition}
        u=0\quad\textrm{ on }\pa\Omega;
      \end{equation}
  \item Navier-slip boundary condition: $\Omega\subset\mathbb{R}^3$ is a bounded and simply connected smooth domain, and
      \begin{equation}\label{Naiver-slip-condition}
        u\cdot n=0,\quad \curl u\times n=0 \textrm{ on }\pa\Omega
      \end{equation}
      \noindent where $n$ is the unit outer normal to $\pa\Omega$. The first condition in \eqref{Naiver-slip-condition} is the non-penetration boundary condition, while the second one is also known as
      \begin{equation*}
        (\mathcal{D}(u)\cdot n)_{\tau}=-\kappa_{\tau}u_{\tau}
      \end{equation*}
      where $\mathcal{D}(u)$ is the deformation tensor:
      \begin{equation*}
        \mathcal{D}(u)=\frac{1}{2}(\nabla u+\nabla u^t),
      \end{equation*}
      and $\kappa_{\tau}$ is the corresponding principal curvature of $\pa\Omega$. Note that $\nabla u$ has the decomposition as
      \begin{equation*}
        \nabla u=\mathcal{D}(u)+\mathcal{S}(u),
      \end{equation*}
      where
      \begin{equation*}
        \mathcal{S}(u)=\frac{1}{2}(\nabla u-\nabla u^t)
      \end{equation*}
  is known as the rigid body rotation tensor.

\end{enumerate}

There are large amounts of literature concerning the large time existence and behavior of solutions to the equations \eqref{CNS-eq}. The one-dimensional problem has been investigated extensively by many people, see for instance \cite{Hoff1987,Kazhikhov1977,Serre1986-1,Serre1986-2} and references therein. The local well-posedness of multidimensional problem was studied by Nash \cite{Nash1962} and Serrin \cite{Serrin1959} in the absence of vacuum, while for the initial density with vacuum, the existence and uniqueness of local strong solution was proved in \cite{Cho2004,Cho2006-1,Cho2006-2,Choe2003,Salvi1993}. Matsumura and Nishida \cite{Matsumura1980} proved the global existence of strong solutions with initial data close to the equilibrium, and later was extended to the discontinuous initial data by Hoff \cite{Hoff1991,Hoff1995}. In the last decade, many important progress on global existence and uniqueness of classical solutions with large oscillations and vacuum to \eqref{CNS-eq} have been made, see for instance \cite{Huang-Li2012,Huang-Li2022,Cai-Li2023,Zhu2024} and references therein. For the existence of solutions for arbitrary large initial data in 3D, the major breakthrough was made by Lions \cite{Lions1998}, where he showed global existence of weak solutions for the whole space, periodic domains, or bounded domains with Dirichlet boundary conditions with $\gamma>\frac{9}{5}$, and later was extended to $\frac{3}{2}$ by Feireisl \cite{Feireisl2001,Feireisl2004,Feireisl2004B}. When the initial data are
assumed to have some spherically symmetric or axisymmetric properties, Jiang and Zhang \cite{Jiang2001}
proved the existence of global weak solutions for any $\gamma>1$.

However, up to now, the uniqueness and regularity for weak solutions in \cite{Lions1998,Feireisl2001} still remains open.
Desjardins in \cite{Desjardins1997} proved higher regularity of local weak solutions of the barotropic compressible
Navier-Stokes equations in dimension two or three under periodic
boundary conditions. It is worth mentioning that one would not expect too much regularity of Lions's weak solutions
in general due to the work of Xin \cite{Xin1998}, where it was shown that there is no global smooth solution to the Cauchy problem of \eqref{CNS-eq} with a nontrivial compactly supported initial density at least in 1D. Later, Xin and Yan \cite{Xin2013} investigated the same problem and proved the smooth solution will blow up in finite time provided that the initial data have an isolated mass group. Recently, Li, Wang and Xin \cite{Li-Xin2019} also gave the ill-posedness of smooth solution to the similar problem. In view of the above results, it is important to study the mechanism of blowup and structure of possible singularities of strong (or classical) solutions to \eqref{CNS-eq} with the initial data allowed to vanish.
For the 3D incompressible Euler equations, Beale-Kato-Majda \cite{Beale1984} presented a sufficient condition for the blowup of
strong solutions, which is the well-known Beale-Kato-Majda blowup criterion, namely,
\begin{equation}\label{BKM-criterion1}
  \lim_{T\rightarrow T^*}\norm{\mathcal{S}(u)}_{L^1(0,T;L^{\infty})}=\infty,
\end{equation}
where $T^*$ is the maximal time of existence of a strong solution and $\mathcal{S}(u)=\frac{1}{2}(\nabla u-\nabla u^t)$
  is the rigid body rotation tensor. Later, Ponce \cite{Ponce1985} gave a different
formulation of the Beale-Kato-Majda blowup criterion in terms of deformation tensor $\mathcal{D}(u)=\frac{1}{2}(\nabla u+\nabla u^t)$, namely,
  \begin{equation}\label{BKM-criterion2}
  \lim_{T\rightarrow T^*}\norm{\mathcal{D}(u)}_{L^1(0,T;L^{\infty})}=\infty.
\end{equation}
Moreover, Constantin \cite{Peter1995} rephrased the Beale-Kato-Majda blowup criterion as follows:
 \begin{equation}\label{BKM-criterion3}
  \lim_{T\rightarrow T^*}\norm{((\nabla u)\xi)\cdot\xi}_{L^1(0,T;L^{\infty})}=\infty,
\end{equation}
where $\xi$ is the unit vector in the direction of vorticity $\textrm{curl} u$.
Recently, Huang-Li-Xin \cite{Huang2011-1} extended the Beale-Kato-Majda type blowup criterion of \cite{Ponce1985} to the 3D barotropic compressible Navier-Stokes equations.
Furthermore, Huang-Li-Xin \cite{Huang2011-1} gave a Serrin type blowup criterion (originated from \cite{Serrin1962}) for the full compressible Navier-Stokes equations, namely,
\begin{equation}\label{Serrin-criterion4}
  \lim_{T\rightarrow T^*}(\norm{\rho}_{L^{\infty}(0,T;L^{\infty})}+\norm{\sqrt{\rho}u}_{L^s(0,T;L^r)})=\infty,
\end{equation}
where
\begin{equation}\label{Serrin-index}
  \frac{2}{s}+\frac{3}{r}\leq1,\quad 3<r\leq\infty.
\end{equation}
These two types of blowup criteria indicate the different directions to control the global regularity of the 3D barotropic compressible Navier-Stokes equations, which have been developed in many different forms. We refer to \cite{Fan2008,Huang2010,Sun2011,Wen2013,Du2014,Wang2020} and references therein.

This paper is strongly motivated by Beale-Kato-Majda \cite{Beale1984} and Huang-Li-Xin \cite{Huang2011-1} and investigates the Beale-Kato-Majda type blowup criterion for the 3D barotropic compressible Navier-Stokes equations in terms of the
the rigid body rotation tensor $\mathcal{S}(u)$.
More precisely, we establish the Beale-Kato-Majda type blowup criterion for the 3D barotropic compressible Navier-Stokes equations
subject to the Cauchy problem, Dirichlet problem and Navier-slip boundary condition in terms of the
the rigid body rotation tensor $\mathcal{S}(u)$.
In particular, this criterion extends the well-known
Beale-Kato-Majda's blowup criterion for the 3D incompressible Euler equations \cite{Beale1984} to the
3D barotropic compressible Navier-Stokes equations, and can be regarded as a complement for the result of Huang-Li-Xin \cite{Huang2011-1}.

Before stating our result, let us introduce the following notations and conventions used throughout this paper. We denote
\begin{equation*}
  \int fdx=\int_{\Omega}fdx.
\end{equation*}
For $1\leq r\leq\infty$, integer $k\geq1$, we denote the standard Sobolev spaces as follows:
\begin{equation*}
  \begin{cases}
     L^r=L^r(\Omega),\,D^{k,r}=\{u\in L_{loc}^1(\Omega): \norm{\nabla^ku}_{L^r}<\infty\},\\
     W^{k,r}=L^r\cap D^{k,r},\,H^k=W^{k,2},\, D^k=D^{k,2},\\
     D_0^1=\{u\in L^6: \norm{\nabla u}_{L^2}<\infty,\textrm{ and \eqref{Cauchy-condition} or \eqref{Dirichlet-condition} or \eqref{Naiver-slip-condition} holds}\},\\
    H_0^1=L^2\cap D_0^1,\,\norm{u}_{D^{k,r}}=\norm{\nabla^ku}_{L^r}.
  \end{cases}
\end{equation*}

As mentioned above, for initial density allowed to vanish, the local existence and uniqueness of strong (or smooth) solutions have been proved by \cite{Cho2004,Cho2006-1,Cho2006-2,Choe2003,Salvi1993}. Before stating their local existence results, let us first give the definition of strong solutions.

\begin{defi}[Strong solutions]
  $(\rho,u)$ is called a strong solution to the system \eqref{CNS-eq} in $\Omega\times(0,T)$ if $(\rho,u)$ satisfies \eqref{CNS-eq} a.e. in $\Omega\times(0,T)$, and for some $q_0\in(3,6]$,
  \begin{equation}\label{strong-sol-defi}
    \begin{aligned}
    &0\leq\rho\in C([0,T], W^{1,q_0}),\, \rho_t\in C([0,T], L^{q_0}),\\
    &u\in C([0,T], D_0^1\cap D^2)\cap L^2(0,T; D^{2,q_0}),\\
    &\sqrt{\rho}u_t\in L^{\infty}(0,T; L^2),\,u_t\in L^2(0,T; D_0^1).
    \end{aligned}
  \end{equation}
\end{defi}
Then Cho \cite{Cho2004} proved the following local result.
\begin{thm}\label{thm-local-existence}
  If the initial data $(\rho_0, u_0)$ satisfies
  \begin{equation}\label{initial-data-condition}
    0\leq\rho_0\in L^1\cap W^{1,\tilde{q}}, \quad u_0\in D_0^1\cap D^2,
  \end{equation}
  for some $\tilde{q}\in(3,\infty)$ and the compatibility condition:
  \begin{equation}\label{compatibility-condition}
    -\mu\Delta u_0-(\lambda+\mu)\nabla\div u_0+\nabla P(\rho_0)=\rho_0^{\frac{1}{2}}g, \textrm{ for some }g\in L^2,
  \end{equation}
  then there exists $T_1\in(0,\infty)$ and a unique strong solution $(\rho,u)$ to the initial boundary value problem \eqref{CNS-eq}\eqref{initial-data} with \eqref{Cauchy-condition} or \eqref{Dirichlet-condition} or \eqref{Naiver-slip-condition} in $\Omega\times(0,T_1]$. Furthermore, the following blowup criterion holds: if $T^*$ is the maximial time of existence of the strong solution $(\rho,u)$ and $T^*<\infty$, then
  \begin{equation}
    \lim_{t\rightarrow T^*}(\norm{\rho}_{H^1\cap W^{1,q_0}}+\norm{u}_{D_0^1})=\infty,
  \end{equation}
  with $q_0=\min\{6,\tilde{q}\}$.
\end{thm}

Now, we are in a position to state our main results.
\begin{thm}\label{thm-blowup}
  Let $(\rho,u)$ be a strong slolution to the initial boundary value problem \eqref{CNS-eq}\eqref{initial-data} with \eqref{Cauchy-condition} or \eqref{Dirichlet-condition} or \eqref{Naiver-slip-condition} satisfying \eqref{strong-sol-defi}. Assume that the initial data $(\rho_0,u_0)$ satisfies \eqref{initial-data-condition} and \eqref{compatibility-condition}. If $T^*$ is the maximal time of existence, then
  \begin{equation}\label{blowup-criterion1}
    \lim_{T\rightarrow T^*}(\norm{\div u}_{L^1(0,T;L^{\infty})}+\norm{\mathcal{S}(u)}_{L^s(0,T;L^r)})=\infty\textrm{ for \eqref{Cauchy-condition}} \textrm{ or } \eqref{Dirichlet-condition},
  \end{equation}
  with
  \begin{equation}\label{Serrin-index2}
    \frac{2}{s}+\frac{3}{r}\leq2, \quad \frac{3}{2}<r\leq\infty;
  \end{equation}
  and
  \begin{equation}\label{blowup-criterion2}
    \lim_{T\rightarrow T^*}(\norm{\div u}_{L^1(0,T;L^{\infty})}+\norm{\mathcal{S}(u)}_{L^1(0,T;L^{\infty})})=\infty\textrm{ for \eqref{Naiver-slip-condition}} .
  \end{equation}
\end{thm}
\begin{rem}
  Theorem \ref{thm-blowup} gives a counterpart of Beale-Kato-Majda type criterion in \cite{Beale1984} for the 3D incompressible Euler equations, which states that singularity can develop only if the size of the divergence or the rotation tensor becomes arbitrarily large.
\end{rem}

\begin{rem}
  It is easy to check that
  \begin{equation}\label{Laplace-identity1}
    \Delta u=\nabla\div u-\nabla\times(\curl u),
  \end{equation}
  and
  \begin{equation}\label{Laplace-identity2}
    \Delta u_j=\pa_i(\mathcal{D}(u)_{ij}+\mathcal{D}(u)_{ji})-\pa_j\mathcal{D}(u)_{ii}.
  \end{equation}
  By the standard elliptic estimates, there exists some equivalence relationship among
  \begin{equation*}
    \nabla u\sim\mathcal{D}(u)\sim(\div u,\mathcal{S}(u))
  \end{equation*}
  in the sense of Sobolev spaces. Thus, it is reasonable and natural to bring forward the blowup criterion in Theorem \ref{thm-blowup}, just like that in \cite{Huang2011-1}. Indeed, for the whole space and the bounded domain with Dirichlet boundary condition, \eqref{blowup-criterion1} can also be adjusted as
   \begin{equation*}
    \lim_{T\rightarrow T^*}\left(\norm{\div u}_{L^1(0,T;L^{\infty})}+\norm{\mathring{\mathcal{D}}(u)}_{L^s(0,T;L^r)}\right)=\infty\textrm{ for \eqref{Cauchy-condition} or \eqref{Dirichlet-condition}}
  \end{equation*}
  with $s,\,r$ satisfying \eqref{Serrin-index2}. Here $\mathring{\mathcal{D}}(u)$ denotes the traceless part of $\mathcal{D}(u)$. For the bounded domain with Navier-slip boundary condition \eqref{Naiver-slip-condition}, one can also get
  \begin{equation*}
    \lim_{T\rightarrow T^*}(\norm{\div u}_{L^1(0,T;L^{\infty})}+\norm{\mathring{\mathcal{D}}(u)}_{L^1(0,T;L^{\infty})})=\infty.
  \end{equation*}
\end{rem}

\begin{rem}
  The condition \eqref{Serrin-index2} is similar to the Serrin's blowup criterion \eqref{Serrin-index} in \cite{Huang2011-2}. The difference between \eqref{blowup-criterion1} and \eqref{blowup-criterion2} comes from the boundary conditions. In particular, the Navier-slip boundary condition prevents us deriving energy estimates of $\norm{\nabla^2u}_{L^2}$ or $\nabla\rho$ directly.
\end{rem}

\begin{rem}
  It seems impossible for us to adapt the method of Huang-Li-Xin \cite{Huang2011-1} to give the blowup criterion for Dirichlet problem as \eqref{blowup-criterion1} or \eqref{blowup-criterion2} without any other ingredients. Indeed, the method of Huang-Li-Xin \cite{Huang2011-1} relies heavily upon $L^\infty$-norm of
  $\mathcal{D}(u)$. Particularly, this enables them to get the uniform  estimates of $\|\nabla \rho\|_{L^2}$ and $\|\nabla u\|_{L^2}$ simultaneously, which play an important role in their analysis.
  Compared to the Cauchy problem and Navier-slip boundary condition
  where the effective viscous flux $G=(2\mu+\lambda)\div u-P$ satisfying the first elliptic equation in \eqref{GW-equation} is used to overcome this difficulty, the new difficulty lies in that there is no boundary information on $G$ for the Dirichlet problem , which prevents us deriving the desired estimate of $G$ from the classic elliptic regularity
  theory. The main observation here, motivated by \cite{Sun2011}, we employ the decomposition $u=h+g$ as different parts with $h$ and $g$ respectively subject to $P$ and $\rho\dot{u}$ to overcome this difficulty. In this process, we also observe a key cancellation which enables us to deal with the trouble terms involving $\nabla P$ and simplifies the estimates on $\norm{\nabla u}_{L^2}$ under the conditions \eqref{Cauchy-condition} and \eqref{Dirichlet-condition}.
\end{rem}

\begin{rem}
  This result can also be extended to the viscous heat-conductive flows and other coupled equations of compressible viscous flows.
\end{rem}

Now, let us illustrate the main difficulties encountered in proving Theorem \ref{thm-blowup}
and explain our strategies to overcome them.
To establish the Beale-Kato-Majda type criterion
stated in proving Theorem \ref{thm-blowup}, along with \cite{Huang2011-1,Huang2011-2}, the main difficulty lies in deriving the $L^{\infty}(0,T;L^q)$-norm of $\nabla \rho$.
It should be noted that the methods of \cite{Huang2011-1,Huang2011-2} depend essentially on the $L^1(0,T;L^{\infty})$-norm of $\mathcal{D}(u)$ or the $L^s(0,T;L^r)$-norm of $u$ heavily, which can provide the direct estimate of $\nabla \rho$ or detour the difficulty of that to directly control the $L^{\infty}(0,T;L^2)$-norm of $\nabla u$. It is thus difficult to adapt their analysis here.
To proceed, we need to develop new ingredients in the proof.
The key step in proving Theorem \ref{thm-blowup} is to derive $L^{\infty}(0,T;L^2)$-norm of $\nabla u$.
Compared to \cite{Huang2011-1,Huang2011-2}, a new difficulty here is that $L^{\infty}(0,T;L^2)$-norm of $\nabla \rho$ is not
available. To overcome this difficulty, we make full use of the dissipation estimates of effective viscous flux $G=(2\mu+\lambda)\div u-P$, and
exploit structure of the equations to avoid the appearances of $\nabla\rho$ or $\nabla P$ in the right hand side of the energy inequality.
We refer to the proofs of \eqref{estimate1-3} and \eqref{estimate1-4} for details.
However, this method is invalid for the Dirichlet problem since
there is no boundary information on $G$ , which prevents us deriving the desired estimate of $G$ from the classic elliptic regularity
  theory as the Cauchy problem and  Navier-slip boundary
condition. Therefore, we must pursue another route by resorting to the velocity decomposition introduced by \cite{Sun2011} to derive $L^{\infty}(0,T;L^2)$-norm of $\nabla u$. In this process, we also observe a key cancellation which enables us to deal with the trouble terms involving $\nabla P$ and simplifies the estimates on $L^{\infty}(0,T;L^2)$-norm of $\nabla u$ under the conditions \eqref{Cauchy-condition} and \eqref{Dirichlet-condition}.
 Then, for the Cauchy problem and Dirichlet problem, we can employ the Beale-Kato-Majda type inequality developed by \cite{Huang-Li2022}
 to get $L^{\infty}(0,T;L^q)$-norm of $\nabla \rho$ by using the $\log$-type Gronwall's argument.
 However, the case of bounded domain with Navier-slip boundary condition is very different due to the boundary effects.
 In particular, the estimates on $\|\sqrt{\rho}u_t\|_{L^2}$ and $\nabla\rho$ cross each other to some extent and cannot directly form the $\log$-type Gronwall's inequality, which requires us to employ new methods. To do this, we first modify the Beale-Kato-Majda's estimate of $\norm{\nabla u}_{L^{\infty}}$ in \cite{Cai-Li2023}. Then, we employ a crucial inequality from \cite{Foias1988}
 to rephrase the energy inequality of $\norm{\nabla\rho}_{L^q}$ to the usual $\log$-type Gronwall's inequality. Thus, this together with the energy inequality on $\|\sqrt{\rho}u_t\|_{L^2}$, we can close the desired a priori estimate.
 We refer to the proofs of Lemma \ref{lem-energy-2-whole2} and Lemma \ref{lem-energy-2-domain} for details.

\section{Proof of Theorem \ref{thm-blowup}}\label{sect-proof-theorem}
Let $(\rho,u)$ be a strong solution to the problem \eqref{CNS-eq}\eqref{initial-data} in Theorem \ref{thm-blowup}. The standard energy estimate on \eqref{CNS-eq} yields that
\begin{equation}\label{energy-0}
  \sup_{0\leq t\leq T}(\norm{\sqrt{\rho}u(t)}_{L^2}^2+\norm{\rho}_{L^{\gamma}}^{\gamma}+\norm{\rho}_{L^1})+\int_0^T\norm{\nabla u}_{L^2}^2dt\leq C_0,\quad 0\leq T<T^*.
\end{equation}
Here we have applied the elliptic estimate on \eqref{Laplace-identity1} for the case \eqref{Naiver-slip-condition}(also see \cite{von-Wahl1992}):
\begin{equation}
  \norm{\nabla u}_{L^2}\lesssim\norm{\div u}_{L^2}+\norm{\curl u}_{L^2}.
\end{equation}
Here and after $\lesssim$ denotes that $a\lesssim b$ means $a\lesssim Cb$ for some generic constant $C>0$.

To prove the theorem, throughout this section, we assume otherwise  that
\begin{equation}\label{assume1}
  \lim_{T\rightarrow T^*}(\norm{\div u}_{L^1(0,T;L^{\infty})}+\norm{\mathcal{S}(u)}_{L^s(0,T;L^r)})\leq M_0\textrm{ for \eqref{Cauchy-condition} or \eqref{Dirichlet-condition}},
\end{equation}
or
\begin{equation}\label{assume2}
    \lim_{T\rightarrow T^*}(\norm{\div u}_{L^1(0,T;L^{\infty})}+\norm{\mathcal{S}(u)}_{L^1(0,T;L^{\infty})})\leq M_0\textrm{ for \eqref{Naiver-slip-condition}}.
\end{equation}
Then applying this assumption to $\eqref{CNS-eq}_1$ yields the $L^{\infty}$ bound of density $\rho$ as follows. Here we omit the proof.

\begin{lem}
  Suppose that
  \begin{equation*}
    \int_0^T\norm{\div u}_{L^{\infty}}dt\leq C, \quad 0\leq T<T^*.
  \end{equation*}
  Then
  \begin{equation}\label{density-bound}
    \sup_{0\leq t\leq T}\norm{\rho}_{L^{\infty}}\leq C, \quad 0\leq T<T^*.
  \end{equation}
  Here and below, $C$ will denote a generic constant depending only on $C_0$, $M_0$, $T$, the initial data and domain $\Omega$.
\end{lem}

Before estimating $\nabla u$, we introduce some notations as follows:
\begin{equation*}
  G:=(2\mu+\lambda)\div u-P, \quad \omega:=\nabla\times u, \quad \dot{f}:=f_t+u\cdot\nabla f
\end{equation*}
which represent the effective viscous flux, vorticity and material derivative of $f$, respectively.
From $\eqref{CNS-eq}_2$, it is easy to check that
\begin{equation}\label{GW-equation}
  \Delta G=\div(\rho\dot{u}),\quad \mu\Delta\omega=\nabla\times(\rho\dot{u}).
\end{equation}
Then the classical elliptic estimates give the following.
\begin{lem}
  Let $\Omega=\mathbb{R}^3$ and $p\in(1,\infty)$. Then
  \begin{equation}\label{GW-estimate-whole}
    \norm{\nabla G}_{L^p}+\norm{\nabla\omega}_{L^p}\lesssim\norm{\rho\dot{u}}_{L^p},\quad \norm{\nabla u}_{L^p}\lesssim\norm{\div u}_{L^p}+\norm{\curl u}_{L^p}.
  \end{equation}
\end{lem}

\begin{lem}
  Let $\Omega$ be a simply connected bounded smooth domain with boundary condition \eqref{Naiver-slip-condition} and $p\in(1,\infty)$. Then
  \begin{equation}\label{G-estimate-domain}
    \norm{\nabla G}_{L^p}+\norm{\nabla\omega}_{L^p}\lesssim\norm{\rho\dot{u}}_{L^p}.
  \end{equation}
\end{lem}
\begin{proof}
  Since $G$ satisfies
  \begin{equation*}
    \begin{cases}
      \Delta G=\div(\rho\dot{u}),\\
      \nabla G\cdot n=\rho\dot{u}\cdot n \quad\mbox{on }\pa\Omega,
    \end{cases}
  \end{equation*}
  where $\nabla\times\omega\cdot n=0$ on $\pa\Omega$ from \cite{Bendali1985}, one gets that for any integer $k\geq 0$ and $p\in(1,\infty)$,
  \begin{equation*}
    \norm{G}_{W^{k+1,p}}\lesssim\norm{\rho\dot{u}}_{W^{k,p}}.
  \end{equation*}
  Then from \eqref{W-estimate1},
  \begin{equation*}
    \norm{\nabla\omega}_{L^p}\lesssim\norm{\nabla\times\omega}_{L^p}
    \lesssim\norm{\nabla G-\rho\dot{u}}_{L^p}\lesssim\norm{\rho\dot{u}}_{L^p}.
  \end{equation*}
\end{proof}

The following lemma comes from \cite{Aramaki2014}.
\begin{lem}
  Let $k\geq 0$ be integer and $p\in(1,\infty)$, and assume $\Omega$ is a simply connected bounded domain in $\mathbb{R}^3$ with $C^{k+1,1}$ boundary $\pa\Omega$. If $v\cdot n=0$ or $v\times n=0$ on $\pa\Omega$, there exists a constant $C(q,k,\Omega)$ such that
  \begin{equation*}
    \norm{v}_{W^{k+1,p}}\leq C(\norm{\div v}_{W^{k,p}}+\norm{\curl v}_{W^{k,p}}).
  \end{equation*}
  In particular, for $k=0$, we have
  \begin{equation}\label{Hodge-decomposition}
    \norm{\nabla v}_{L^p}\leq C(\norm{\div v}_{L^p}+\norm{\curl v}_{L^p}).
  \end{equation}
\end{lem}

Then we get the estimate on $\nabla u$ in the following lemma.
\subsection{Estimates on $\nabla u$ for the case \eqref{Cauchy-condition}  or \eqref{Dirichlet-condition} or \eqref{Naiver-slip-condition}.}
\begin{lem}\label{lem-energy-1}
Under the condition \eqref{Cauchy-condition} or \eqref{Naiver-slip-condition}, it holds that for any $T<T^*$,
\begin{equation}\label{energy-1}
   \sup_{0\leq t\leq T}\norm{\nabla u}_{L^2}^2+\int_0^T(\norm{\sqrt{\rho}\dot{u}}_{L^2}^2+\norm{(\nabla G,\nabla\omega)}_{L^2}^2)dt\leq C.
\end{equation}
Under the condition \eqref{Dirichlet-condition}, it holds that for any $T<T^*$,
\begin{equation}\label{energy-1-D}
   \sup_{0\leq t\leq T}\norm{\nabla u}_{L^2}^2+\int_0^T\norm{\sqrt{\rho}\dot{u}}_{L^2}^2dt\leq C.
\end{equation}
\end{lem}

\begin{proof}[Proof of Lemma \ref{lem-energy-1}]
  Multiplying $\rho^{-1}(\mu\Delta u+(\mu+\lambda)\nabla\div u-\nabla P)$ on $\eqref{CNS-eq}_2$ and integrating over $\Omega$, we have
  \begin{equation}\label{estimate1-0}
  \begin{aligned}
    &\quad\frac{d}{dt}\int(\frac{\mu}{2}|\curl u|^2+\frac{2\mu+\lambda}{2}(\div u)^2)dx+\int\rho^{-1}(\mu\Delta u+(\mu+\lambda)\nabla\div u-\nabla P)^2dx\\
    &=-\mu\int u\cdot\nabla u\cdot\nabla\times\curl udx+\int u\cdot\nabla u\cdot\nabla Gdx-\int u_t\cdot\nabla Pdx
  \end{aligned}
  \end{equation}
  due to $\Delta u=\nabla\div u-\nabla\times\curl u$.
  For the boundary condition \eqref{Cauchy-condition} or \eqref{Naiver-slip-condition},
  \begin{equation}\label{estimate1-1}
    \int\rho|\dot{u}|^2=\int\rho^{-1}(\mu\Delta u+(\mu+\lambda)\nabla\div u-\nabla P)^2dx\geq C^{-1}(\norm{\nabla G}_{L^2}^2+\norm{\mu\nabla\times\curl u}_{L^2}^2),
  \end{equation}
  due to $\rho^{-1}\geq C^{-1}>0$.

  Next we treat each term on the righthand side of \eqref{estimate1-0} under the boundary conditions \eqref{Dirichlet-condition} and \eqref{Naiver-slip-condition}, for the reason that the case \eqref{Cauchy-condition} is much simpler than the former cases due to the absence of boundary.

  First, we consider the boundary condition \eqref{Naiver-slip-condition} and \eqref{Cauchy-condition}.
  By \eqref{Naiver-slip-condition} and the identities that
  \begin{equation*}
    \begin{aligned}
    (\nabla\times a)\times b&=b\cdot\nabla a-\nabla a\cdot b,\\
    \nabla\times(a\times b)&=a\div b+b\cdot\nabla a-(b\div a+a\cdot\nabla b),
    \end{aligned}
  \end{equation*}
  then after integrating by parts one can get
  \begin{equation}\label{estimate1-2}
    \begin{aligned}
    &\quad\int u\cdot\nabla u\cdot\nabla\times\curl udx=\int\curl u\cdot\nabla\times(u\cdot\nabla u)dx\\
    &=\int\curl u\cdot\nabla\times(\curl u\times u)dx\\
    &=\frac{1}{2}\int|\curl u|^2\div udx-\int\curl u\cdot\nabla u\cdot\curl udx\\
    &\lesssim\norm{\curl u}_{L^r}\norm{\nabla u}_{L^{\frac{2r}{r-1}}}^2\lesssim\norm{\curl u}_{L^r}\norm{\nabla u}_{L^2}^{2-\frac{3}{r}}\norm{\nabla u}_{L^6}^{\frac{3}{r}}\\
    &\leq C(\epsilon)(1+\norm{\curl u}_{L^r}^s)\norm{\nabla u}_{L^2}^2+\epsilon(\norm{P}_{L^6}^2+\norm{G}_{L^6}^2+\norm{\omega}_{L^6}^2)\\
    &\leq C(\epsilon)(1+\norm{\curl u}_{L^r}^s)\norm{\nabla u}_{L^2}^2+\epsilon(\norm{\nabla G}_{L^2}^2+\norm{\nabla\omega}_{L^2}^2)+C
    \end{aligned}
  \end{equation}
  due to \eqref{GW-estimate-whole} or \eqref{Hodge-decomposition} and the bound of $\rho$. Here $\epsilon$ is sufficiently small, and \eqref{Serrin-index2} holds.

  For the case \eqref{Cauchy-condition},
  \begin{equation*}
    \norm{(\nabla G, \nabla\omega)}_{L^2}\leq C\norm{\rho\dot{u}}_{L^2}\leq C\norm{\sqrt{\rho}\dot{u}}_{L^2}.
  \end{equation*}

  For another case \eqref{Naiver-slip-condition}, since $\omega$ satifies $\omega\times n=0$ on $\pa\Omega$ and $\div\omega=0$, one can get from \eqref{Hodge-decomposition} that
  \begin{equation}\label{W-estimate1}
    \norm{\nabla\omega}_{L^p}\lesssim \norm{\nabla\times\omega}_{L^p}
  \end{equation}
  for any $p\in(1,\infty)$.

  Thus,
  \begin{equation}\label{estimate1-2-1}
  \begin{aligned}
    &\quad\int u\cdot\nabla u\cdot\nabla\times\curl udx\\
    &\leq C(\epsilon)(1+\norm{\curl u}_{L^r}^s)\norm{\nabla u}_{L^2}^2+\epsilon(\norm{\sqrt{\rho}\dot{u}}_{L^2}^2+\norm{\nabla G}_{L^2}^2+\norm{\nabla\times\omega}_{L^2}^2)+C.
  \end{aligned}
  \end{equation}
  Similarly, for the second term,
  \begin{equation}\label{estimate1-3}
    \begin{aligned}
    &\quad\int u\cdot\nabla u\cdot\nabla Gdx\\
    &=\int_{\pa\Omega}u\cdot\nabla u\cdot n GdS-\int\nabla u:\nabla u^tGdx-\int u\cdot\nabla\div uGdx\\
    &=-\int_{\pa\Omega}u\cdot\nabla n\cdot uGdS-\int\nabla u:\nabla u^t Gdx+\frac{1}{2(2\mu+\lambda)}\int\div uG^2dx\\
    &\quad+\frac{1}{2\mu+\lambda}\int(\div uPG+Pu\cdot\nabla G)dx\\
    &\lesssim \norm{G}_{L^2(\pa\Omega)}\norm{u}_{L^4(\pa\Omega)}^2+\norm{G}_{L^{\infty}}\norm{\nabla u}_{L^2}^2+\norm{\div u}_{L^{\infty}}(\norm{G}_{L^2}^2+1)\\
    &\quad+\epsilon\norm{\nabla G}_{L^2}^2+C(\epsilon)\norm{\sqrt{\rho}u}_{L^2}^2\\
    &\leq C\norm{u}_{L^4}^{\frac{3}{2}}\norm{\nabla u}_{L^4}^{\frac{1}{2}}\norm{G}_{L^2}^{\frac{1}{2}}\norm{\nabla G}_{L^2}^{\frac{1}{2}}+C(\norm{\nabla u}_{L^2}^2+1)(\norm{\div u}_{L^{\infty}}+1)+\epsilon\norm{\nabla G}_{L^2}^2+C(\epsilon)\\
    &\leq C\norm{\nabla u}_{L^2}^{\frac{3}{2}}(\norm{(\nabla G, \nabla\omega)}_{L^2}+\norm{P}_{L^4})^{\frac{1}{2}}(\norm{\nabla u}_{L^2}+\norm{P}_{L^2})^{\frac{1}{2}}\norm{\nabla G}_{L^2}^{\frac{1}{2}}\\
    &\quad+C(\norm{\nabla u}_{L^2}^2+1)(\norm{\div u}_{L^{\infty}}+1)+\epsilon\norm{\nabla G}_{L^2}^2+C(\epsilon)\\
    &\leq C(\norm{\nabla u}_{L^2}^2+1)(\norm{(\nabla G,\nabla\times\omega)}_{L^2}+1)+C(\epsilon)(\norm{\nabla u}_{L^2}^2+1)(\norm{\div u}_{L^{\infty}}+1)+\epsilon\norm{\nabla G}_{L^2}^2\\
    &\leq \epsilon\norm{(\nabla G,\nabla\times\omega)}_{L^2}^2+C(\epsilon)(\norm{\nabla u}_{L^2}^2+1)(\norm{\div u}_{L^{\infty}}+\norm{\nabla u}_{L^2}^2+1)
    \end{aligned}
  \end{equation}
  where we have used \eqref{Hodge-decomposition}, \eqref{W-estimate1}, \eqref{energy-0} the bound of density $\rho$ and the following trace inequality (see Theorem 1.6.6 in \cite{Brenner2008}) : Suppose that $\Omega$ has a Lipschitz boundary and $p\in[1,\infty]$. Then there is a constant $C$ such that
  \begin{equation}\label{trace-inequality}
    \norm{v}_{L^p(\pa\Omega)}\leq C\norm{v}_{L^p(\Omega)}^{1-\frac{1}{p}}\norm{v}_{W^{1,p}(\Omega)}^{\frac{1}{p}}, \quad \forall v\in W^{1,p}(\Omega).
  \end{equation}
  By analogy,
  \begin{equation}\label{estimate1-4}
    \begin{aligned}
    &\quad-\int u_t\cdot\nabla Pdx\\
    &=\frac{d}{dt}\int P\div udx-\int P_t\div udx\\
    &=\frac{d}{dt}\int P\div udx+\int \div(uP)\div udx+(\gamma-1)\int P(\div u)^2dx\\
    &=\frac{d}{dt}\int P\div udx-\frac{1}{2\mu+\lambda}\int Pu\cdot\nabla(G+P)dx+(\gamma-1)\int P(\div u)^2dx\\
    &\leq \frac{d}{dt}\int P\div udx+\epsilon\norm{\nabla G}_{L^2}^2+C(\epsilon)(1+\norm{\nabla u}_{L^2}^2).
    \end{aligned}
  \end{equation}
  Plugging \eqref{estimate1-1}, \eqref{W-estimate1}, \eqref{estimate1-2-1}, \eqref{estimate1-3} and \eqref{estimate1-4} into \eqref{estimate1-0} yields that for sufficiently small $\epsilon$,
  \begin{equation}\label{estimate1-0-1}
  \begin{aligned}
    &\quad\frac{d}{dt}\int(\frac{\mu}{2}|\curl u|^2+\frac{2\mu+\lambda}{2}(\div u)^2-P\div u)dx+C_0\norm{(\sqrt{\rho}\dot{u}, \nabla G, \nabla\omega)}_{L^2}^2\\
    &\leq C(\norm{\nabla u}_{L^2}^2+1)(\norm{\curl u}_{L^r}^s+\norm{\div u}_{L^{\infty}}+\norm{\nabla u}_{L^2}^2+1).
  \end{aligned}
  \end{equation}
  Then combining \eqref{energy-0} and applying Gronwall's inequality, we deduce that \eqref{energy-1} holds.

  Now we consider the case \eqref{Dirichlet-condition}, which is slightly different from the boundary condition \eqref{Cauchy-condition} and \eqref{Naiver-slip-condition}. Note that in \eqref{estimate1-2} it suffices to control $\norm{\nabla u}_{L^6}$. Here we introduce the decomposition $u=h+g$ in \cite{Sun2011} such that
  \begin{equation}\label{u-decomposition-eq1}
    \begin{cases}
      Lh:=\mu\Delta h+(\mu+\lambda)\nabla\div h=\nabla P,\\
      h=0,\mbox{ on }\pa\Omega,
    \end{cases}
  \end{equation}
  and
  \begin{equation}\label{u-decomposition-eq2}
    \begin{cases}
      Lg=\rho\dot{u},\\
      g=0,\mbox{ on }\pa\Omega.
    \end{cases}
  \end{equation}
  Then the following elliptic estimates raised later in \eqref{Lame-estimate} holds: for any $p\in(1,\infty)$,
  \begin{equation}\label{u-decomposition-estimate}
    \begin{aligned}
    \norm{h}_{W^{1,p}}\leq C\norm{P}_{L^p},\quad\norm{g}_{W^{2,p}}\leq C\norm{\rho\dot{u}}_{L^p}.
    \end{aligned}
  \end{equation}
  Then,
  \begin{equation}\label{u-estimate-L6}
    \norm{\nabla u}_{L^6}\leq C(\norm{\nabla h}_{L^6}+\norm{\nabla g}_{L^6})\leq C(\norm{P}_{L^6}+\norm{\rho\dot{u}}_{L^2})\leq C(1+\norm{\sqrt{\rho}\dot{u}}_{L^2}).
  \end{equation}
  With that, we can modify \eqref{estimate1-2} into
  \begin{equation}\label{estimate1-2-D}
    \int u\cdot\nabla u\cdot\nabla\times\curl udx\leq C(\epsilon)(1+\norm{\curl u}_{L^r}^s)\norm{\nabla u}_{L^2}^2+\epsilon(1+\norm{\sqrt{\rho}\dot{u}}_{L^2}).
  \end{equation}
  Also, through integration by part we obtain
  \begin{equation}\label{estimate1-3-D}
    \begin{aligned}
    &\quad\int u\cdot\nabla u\cdot\nabla Gdx-\int u_t\cdot\nabla Pdx\\
    &=-(2\mu+\lambda)\int\nabla u:\nabla u^t\div udx+\frac{2\mu+\lambda}{2}\int(\div u)^3dx+\int\nabla u:\nabla u^tPdx\\
    &\quad+\int Pu\cdot\nabla\div udx+\frac{d}{dt}\int P\div udx-\int P_t\div udx\\
    &=-(2\mu+\lambda)\int\nabla u:\nabla u^t\div udx+\frac{2\mu+\lambda}{2}\int(\div u)^3dx+\int\nabla u:\nabla u^tPdx\\
    &\quad-\int\div(Pu)\div udx+\frac{d}{dt}\int P\div udx+\int\div(Pu)\div udx+(\gamma-1)\int P(\div u)^2dx\\
    &\leq \frac{d}{dt}\int P\div udx+C(\norm{\div u}_{L^{\infty}}+1)\norm{\nabla u}_{L^2}^2.
    \end{aligned}
  \end{equation}
  Thus, by combining \eqref{estimate1-2-D}, \eqref{estimate1-3-D}, \eqref{Hodge-decomposition} with \eqref{energy-0}, the Gronwall's inequality yields the desired estimate \eqref{energy-1-D}.
  The proof of Lemma \ref{lem-energy-1} is completed.
\end{proof}

\begin{rem}
  The estimates in Lemma \ref{lem-energy-1} holds for assumption \eqref{assume1}, not only in Cauchy problem \eqref{Cauchy-condition} and Dirichlet problem  \eqref{Naiver-slip-condition}. The difference of the assumptions \eqref{assume1} and \eqref{assume2} comes from the estimates on $\nabla P$ or $\nabla\rho$. The Navier-slip boundary condition is more complex than the Cauchy case.
\end{rem}

\begin{rem}
  The difference of estimating on $\norm{\nabla u}_{L^2}$ between the case \eqref{Naiver-slip-condition} and \eqref{Dirichlet-condition} is a direct consequence of  boundary effects. On one hand, by noticing that $u\cdot\nabla u\cdot n\neq0$ on $\pa\Omega$ in \eqref{estimate1-3}, one need use the dissipation $\nabla G$ to deal with boundary terms. On the other hand, the velocity decomposition in \cite{Sun2011} dose not work since one cannot get the estimate on $\nabla G$ subject to $\rho\dot{u}$ just like \eqref{G-estimate-domain} from the elliptic regularity theory under the Dirichlet boundary condition \eqref{Dirichlet-condition}.
\end{rem}

\begin{rem}
  Under the condition \eqref{Dirichlet-condition}, there exists a key cancellation in the estimate \eqref{estimate1-3-D}, which enables us to deal with the trouble terms involving $\nabla P$. Of course, we can employ the velocity decomposition $u=h+g$ to estimate the trouble term $\int P_t\div udx$ by as follows:
  \begin{equation*}
    \begin{aligned}
    &\quad-\int P_t\div udx\\
    &=\int Lh_t\cdot hdx+\int\div(Pu)\div gdx+(\gamma-1)\int P\div u\div gdx\\
    &=-\frac{d}{dt}\int(\frac{\mu}{2}|\nabla h|^2+\frac{\mu+\lambda}{2}(\div h)^2)dx-\int Pu\cdot\nabla\div gdx+(\gamma-1)\int P\div u\div(u-h)dx\\
    &\leq -\frac{d}{dt}\int(\frac{\mu}{2}|\nabla h|^2+\frac{\mu+\lambda}{2}(\div h)^2)dx+C\norm{\sqrt{\rho}u}_{L^2}\norm{\nabla^2g}_{L^2}+C(\norm{\nabla u}_{L^2}^2+\norm{\nabla h}_{L^2}^2)\\
    &\leq -\frac{d}{dt}\int(\frac{\mu}{2}|\nabla h|^2+\frac{\mu+\lambda}{2}(\div h)^2)dx+\epsilon\norm{\sqrt{\rho}\dot{u}}_{L^2}^2+C(\epsilon)(\norm{\sqrt{\rho}u}_{L^2}^2+\norm{\nabla u}_{L^2}^2+\norm{P}_{L^2}^2).
    \end{aligned}
  \end{equation*}
\end{rem}

Next we turn to estimate $\nabla^2 u$ and $\nabla\rho$. For the cases \eqref{Cauchy-condition} or \eqref{Dirichlet-condition} and \eqref{Naiver-slip-condition}, we apply different methods of estimating.

\subsection{Estimates on $\nabla^2 u$ for $\Omega=\mathbb{R}^3$ or Dirichlet boundary condition \eqref{Dirichlet-condition}.}
\begin{lem}\label{lem-energy-2-whole1}
  Under the condition \eqref{Cauchy-condition}, it holds that for $0\leq T<T^*$,
  \begin{equation}\label{energy-2-whole}
    \sup_{0\leq t\leq T}\norm{\sqrt{\rho}\dot{u}}_{L^2}^2+\int_0^T(\norm{\nabla\dot{u}}_{L^2}^2+\norm{(\div u,\omega)}_{L^{\infty}}^2)dt\leq C.
  \end{equation}
  Under the condition \eqref{Dirichlet-condition}, it holds that for $0\leq T<T^*$,
  \begin{equation}\label{energy-2-D}
    \sup_{0\leq t\leq T}\norm{\sqrt{\rho}\dot{u}}_{L^2}^2+\int_0^T\norm{\nabla\dot{u}}_{L^2}^2dt\leq C.
  \end{equation}
\end{lem}

\begin{proof}
  Due to the compatibility condition \eqref{compatibility-condition}, we can define $\sqrt{\rho}\dot{u}|_{t=0}=g\in L^2$. Applying $\dot{u}^j[\pa_t+\div(u\cdot)]$ to the $j$th component of $\eqref{CNS-eq}_2$ and integrating by parts yields
  \begin{equation}\label{whole-estimate2-0}
  \begin{aligned}
    \frac{d}{dt}\int\frac{1}{2}\rho|\dot{u}|^2dx&=\int[-\dot{u}^j[\pa_jP_t+\div(\pa_jPu)]+\mu\dot{u}^j[\Delta u_t^j+\div(u\Delta u^j)]\\
    &\quad+(\mu+\lambda)\dot{u}^j[\pa_j\div u_t+\div(u\pa_j\div u)]]dx\\
    &=\sum_{i=1}^3A_i.
  \end{aligned}
  \end{equation}
  Then integrating by parts gives that
  \begin{equation}\label{whole-estimate2-1}
    \begin{aligned}
      A_1&=\int(\div\dot{u}P_t+u\cdot\nabla\dot{u}\cdot\nabla P)dx\\
      &=\int(-\gamma\div\dot{u}P\div u-\div\dot{u}u\cdot\nabla P+u\cdot\nabla\dot{u}\cdot\nabla P)dx\\
      &=\int(-\gamma\div\dot{u}P\div u+P\div(\div\dot{u}u)-P\div(u\cdot\nabla\dot{u}))dx\\
      &\lesssim\norm{\nabla u}_{L^2}\norm{\nabla\dot{u}}_{L^2}\\
      &\leq\epsilon\norm{\nabla\dot{u}}_{L^2}^2+C(\epsilon).
    \end{aligned}
  \end{equation}
  Similarly,
  \begin{equation}\label{whole-estimate2-2}
    \begin{aligned}
    A_2&=-\mu\int(\nabla\dot{u}:\nabla u_t+u\cdot\nabla\dot{u}\cdot\Delta u)dx\\
    &=-\mu\int(|\nabla\dot{u}|^2-\nabla\dot{u}:\nabla u\cdot\nabla u-\nabla\dot{u}:(u\cdot\nabla)\nabla u+u\cdot\nabla\dot{u}\cdot\Delta u)dx\\
    &=-\mu\int(|\nabla\dot{u}|^2-\nabla\dot{u}:\nabla u\cdot\nabla u+\div u\nabla\dot{u}:\nabla u+u\cdot\nabla(\nabla\dot{u}):\nabla u+u\cdot\nabla\dot{u}\cdot\Delta u)dx\\
    &=-\mu\int(|\nabla\dot{u}|^2-\nabla\dot{u}:\nabla u\cdot\nabla u+\div u\nabla\dot{u}:\nabla u-\nabla u\cdot\nabla\dot{u}:\nabla u)dx\\
    &\leq-\frac{\mu}{2}\norm{\nabla\dot{u}}_{L^2}^2+C\norm{\nabla u}_{L^4}^4.
    \end{aligned}
  \end{equation}
  Also,
  \begin{equation}\label{whole-estimate2-3}
    \begin{aligned}
    A_3&=-(\mu+\lambda)\int(\div\dot{u}\div u_t+u\cdot\nabla\dot{u}\cdot\nabla\div u)dx\\
    &=-(\mu+\lambda)\int((\div\dot{u})^2-\div\dot{u}\div(u\cdot\nabla u)+u\cdot\nabla\dot{u}\cdot\nabla\div u)dx\\
    &=-(\mu+\lambda)\int((\div\dot{u})^2-\div\dot{u}\nabla u:\nabla u^t+\div\dot{u}(\div u)^2+u\cdot\nabla\div\dot{u}\div u+u\cdot\nabla\dot{u}\cdot\nabla\div u)dx\\
    &=-(\mu+\lambda)\int((\div\dot{u})^2-\div\dot{u}\nabla u:\nabla u^t+\div\dot{u}(\div u)^2-\div u\nabla u:\nabla\dot{u}^t)dx\\
    &\leq-\frac{\mu+\lambda}{2}\norm{\div\dot{u}}_{L^2}^2+\epsilon\norm{\nabla\dot{u}}_{L^2}^2
    +C(\epsilon)\norm{\nabla u}_{L^4}^4.
    \end{aligned}
  \end{equation}
  Then for the case \eqref{Cauchy-condition}, inserting \eqref{whole-estimate2-1}-\eqref{whole-estimate2-3} into \eqref{whole-estimate2-0} yields
  \begin{equation}\label{whole-estimate2-0-1}
  \begin{aligned}
    &\quad\frac{d}{dt}\norm{\sqrt{\rho}\dot{u}}_{L^2}^2+C_0\norm{\nabla\dot{u}}_{L^2}^2\\
    &\leq C(\norm{\nabla u}_{L^4}^4+1)\leq C(\norm{\nabla u}_{L^2}\norm{\nabla u}_{L^6}^3+1)\\
    &\leq C(\norm{(G, \omega)}_{L^6}^3+\norm{P}_{L^6}^3+1)\\
    &\leq C(\norm{(\nabla G, \nabla\omega)}_{L^2}^3+1)\leq C(\norm{\rho\dot{u}}_{L^2}^3+1)\\
    &\leq C(\norm{\sqrt{\rho}\dot{u}}_{L^2}^3+1).
  \end{aligned}
  \end{equation}
  Then combining \eqref{energy-1} and applying the Gronwall's inequality gives that
  \begin{equation}\label{energy-2-whole1}
    \sup_{0\leq t\leq T}\norm{\sqrt{\rho}\dot{u}}_{L^2}^2+\int_0^T\norm{\nabla\dot{u}}_{L^2}^2dt\leq C.
  \end{equation}
  Therefore,
  \begin{equation}
    \begin{aligned}
    \int_0^T\norm{(\div u,\omega)}_{L^{\infty}}^2dx&\leq C\int_0^T(\norm{(G,\omega)}_{L^2}^2+\norm{(\nabla G,\nabla\omega)}_{L^6}^2+1)dt\\
    &\leq C\int_0^T(\norm{(G,\omega)}_{L^2}^2+\norm{\rho\dot{u}}_{L^6}^2+1)dt\\
    &\leq C\int_0^T(1+\norm{\nabla\dot{u}}_{L^2}^2)dt\leq C.
    \end{aligned}
  \end{equation}
  which together with \eqref{energy-2-whole1} gives \eqref{energy-2-whole}.

  For the another case \eqref{Dirichlet-condition}, due to \eqref{u-estimate-L6} one can modify the estimate \eqref{whole-estimate2-0-1} as
  \begin{equation}\label{Dirichlet-estimate2}
    \frac{d}{dt}\norm{\sqrt{\rho}\dot{u}}_{L^2}^2+C_0\norm{\nabla\dot{u}}_{L^2}^2\\
    \leq C(\norm{\nabla u}_{L^2}\norm{\nabla u}_{L^6}^3+1)\leq C(1+\norm{\sqrt{\rho}\dot{u}}_{L^2}^3).
  \end{equation}
  Then with \eqref{energy-1-D}, the Gronwall's inequality yields \eqref{energy-2-D}.

\end{proof}

Before the next step, we recall the Lam\'{e}'s system (see \cite{Cai-Li2023}):
\begin{equation}\label{Lame-system}
  \begin{cases}
    \mu\Delta u+(\mu+\lambda)\nabla\div u=f,\\
    \mbox{\eqref{Cauchy-condition} or \eqref{Dirichlet-condition} or \eqref{Naiver-slip-condition} holds}.
  \end{cases}
\end{equation}
Then for any $p\in(1,\infty)$ and integer $k\geq 0$, there exists a constant $C$ such that
\begin{equation}\label{Lame-estimate}
\begin{aligned}
\norm{u}_{W^{k+2,p}}&\leq C(\norm{f}_{W^{k,p}}+\norm{u}_{L^p}), \\
\norm{u}_{W^{k+1,p}}&\leq C(\norm{g}_{W^{k,p}}+\norm{u}_{L^p}), \mbox{ for }f=\nabla g.
\end{aligned}
\end{equation}
In particular, for the case \eqref{Dirichlet-condition} and \eqref{Naiver-slip-condition}, $\norm{u}_{L^p}$ on the righthand side of inequalities can be removed. For the case \eqref{Cauchy-condition}, the disappearance of $\norm{u}_{L^p}$ only occurs at the estimate of $\norm{\nabla^2u}_{L^p}$.

The next lemma is used to bound the density gradient.

\begin{lem}\label{lem-energy-2-whole2}
  Under the condition \eqref{Cauchy-condition} or \eqref{Dirichlet-condition}, for any $q\in(3,\min\{6,\tilde{q}\}]$ with $\tilde{q}$ in \eqref{initial-data-condition}, it holds that for $0\leq T<T^*$,
  \begin{equation*}
    \sup_{0\leq t\leq T}(\norm{\rho}_{W^{1,q}}+\norm{\nabla^2u}_{L^2})\leq C.
  \end{equation*}
\end{lem}

\begin{proof}
It is easy to check that,
\begin{equation*}
\begin{aligned}
  &(|\nabla\rho|^q)_t+\div(|\nabla\rho|^qu)+(q-1)|\nabla\rho|^q\div u\\
  &\quad+q|\nabla\rho|^{q-2}(\nabla\rho)^t\nabla u(\nabla\rho)+q\rho|\nabla\rho|^{q-2}\nabla\rho\cdot\nabla\div u=0,
\end{aligned}
\end{equation*}
which yields that for the case \eqref{Cauchy-condition} or \eqref{Naiver-slip-condition},
\begin{equation}\label{gradient-density}
\begin{aligned}
    \frac{d}{dt}\norm{\nabla\rho}_{L^q}&\leq C(\norm{\nabla u}_{L^{\infty}}+\norm{\div u}_{L^{\infty}}+1)\norm{\nabla\rho}_{L^q}+C\norm{\nabla G}_{L^q}\\
    &\leq C(\norm{\nabla u}_{L^{\infty}}+\norm{\div u}_{L^{\infty}}+1)\norm{\nabla\rho}_{L^q}+C\norm{\rho\dot{u}}_{L^q}.
\end{aligned}
\end{equation}
  Applying \eqref{Lame-estimate} to $\eqref{CNS-eq}_2$ gives that
  \begin{equation}\label{gradient-u-2rd-whole}
    \begin{aligned}
    \norm{\nabla^2 u}_{L^q}&\lesssim\norm{\rho\dot{u}}_{L^q}+\norm{\nabla P}_{L^q}\\
    &\lesssim\norm{\rho\dot{u}}_{L^2}+\norm{\rho\dot{u}}_{L^6}+\norm{\nabla\rho}_{L^q}\\
    &\leq C(1+\norm{\nabla\dot{u}}_{L^2}+\norm{\nabla\rho}_{L^q}),
    \end{aligned}
  \end{equation}
  which, combined with the Beale-Kato-Majda type inequality (see \cite{Huang2011-2}), leads to that for the case \eqref{Cauchy-condition},
  \begin{equation}\label{gradient-u-infty1}
  \begin{aligned}
    \norm{\nabla u}_{L^{\infty}}&\leq C(\norm{\div u}_{L^{\infty}}+\norm{\omega}_{L^{\infty}})\ln(e+\norm{\nabla^2u}_{L^q})+C\norm{\nabla u}_{L^2}+C\\
    &\leq C(\norm{\div u}_{L^{\infty}}+\norm{\omega}_{L^{\infty}})\ln(e+\norm{\nabla\dot{u}}_{L^2})\\
    &\quad +C(\norm{\div u}_{L^{\infty}}+\norm{\omega}_{L^{\infty}})\ln(e+\norm{\nabla\rho}_{L^q})+C.
  \end{aligned}
  \end{equation}
  Set
  \begin{equation*}
    f(t):=e+\norm{\nabla\rho}_{L^q},\quad g(t):=(\norm{(\div u,\omega)}_{L^{\infty}}+\norm{\nabla\dot{u}}_{L^{2}}+1)\ln(e+\norm{\nabla\dot{u}}_{L^2}).
  \end{equation*}
  It follows from \eqref{gradient-density} and \eqref{gradient-u-infty1} that
  \begin{equation}\label{Gronwall-inequality}
    f'(t)\leq Cg(t)f(t)+Cg(t)f(t)\ln f(t)+Cg(t)\qquad\mbox{($\log$-type Gronwall's inequality)}
  \end{equation}
  which gives
  \begin{equation*}
    (\ln f(t))'\leq Cg(t)+g(t)\ln f(t)
  \end{equation*}
  by $f(t)>1$. The Gronwall's inequality and \eqref{energy-2-whole} yields
  \begin{equation*}
    \sup_{0\leq t\leq T}f(t)\leq C.
  \end{equation*}
  Then similar to \eqref{gradient-u-2rd-whole} and by Gagliardo-Nirenberg's inequality in \cite{Nirenberg1959}, we have for $0\leq t\leq T$,
  \begin{equation*}
    \norm{\nabla^2u}_{L^2}\leq C(\norm{\rho\dot{u}}_{L^2}+\norm{\nabla\rho}_{L^2})\leq C(\norm{\rho\dot{u}}_{L^2}+\norm{\nabla\rho}_{L^q}+\norm{\rho}_{L^1})\leq C.
  \end{equation*}

  For the case \eqref{Dirichlet-condition}, we need a new estimate for $\norm{\nabla u}_{L^{\infty}}$ as follows: for the decomposition $u=h+g$,
  \begin{equation}\label{gradient-u-infty-D}
  \begin{aligned}
    \norm{\nabla u}_{L^{\infty}}&\leq \norm{\nabla h}_{L^{\infty}}+\norm{\nabla g}_{L^{\infty}}\\
    &\leq C(1+\norm{\nabla h}_{BMO}\ln(e+\norm{\nabla^2h}_{L^q}))+C\norm{g}_{W^{2,q}}\\
    &\leq C\norm{P}_{L^{\infty}\cap L^2}\ln(e+\norm{\nabla^2u}_{L^q}+\norm{\nabla^2g}_{L^q})+C\norm{\rho\dot{u}}_{L^q}+C\\
    &\leq C\ln(e+\norm{\nabla^2u}_{L^q}+\norm{\nabla\dot{u}}_{L^2})+C\norm{\nabla\dot{u}}_{L^2}+C,
    \end{aligned}
  \end{equation}
  where we have used \eqref{u-decomposition-estimate} and the inequalities from \cite{Sun2011} (See Proposition 2.2 and Lemma 2.3):
  \begin{align}\label{new-inequality}
    \norm{f}_{L^{\infty}}&\leq C\norm{ f}_{BMO}\ln(e+\norm{\nabla f}_{L^q})+C, \textrm{ for any }f\in W^{1,p}(\Omega), \\
    \norm{\nabla h}_{BMO}&\leq C(\norm{P}_{L^{\infty}}+\norm{P}_{L^2}), \textrm{ if $h$ satisfies \eqref{u-decomposition-eq1}}.
  \end{align}
  Then combined with \eqref{gradient-u-2rd-whole}, we modify the estimate \eqref{gradient-density} as
  \begin{equation}
    \begin{aligned}
    \frac{d}{dt}\norm{\nabla\rho}_{L^q}&\leq C\norm{\nabla u}_{L^{\infty}}\norm{\nabla\rho}_{L^q}+C\norm{\nabla^2u}_{L^q}\\
    &\leq C\norm{\nabla\rho}_{L^q}\ln(e+\norm{\nabla\rho}_{L^q}+\norm{\nabla\dot{u}}_{L^2})+C(1+\norm{\nabla\rho}_{L^q}+\norm{\nabla\dot{u}}_{L^2})\\
    &\leq C\norm{\nabla\rho}_{L^q}\ln(e+\norm{\nabla\rho}_{L^q})+\norm{\nabla\rho}_{L^q}\ln(e+\norm{\nabla\dot{u}}_{L^2})+C(1+\norm{\nabla\dot{u}}_{L^2}).
    \end{aligned}
  \end{equation}
  The log-type Gronwall's inequality as \eqref{Gronwall-inequality} yields
  \begin{equation*}
    \sup_{0\leq t\leq T}\norm{\nabla\rho}_{L^q}\leq C,
  \end{equation*}
  and the bound of $\norm{\nabla^2u}_{L^2}$ follows immediately.
 Thus Lemma \ref{lem-energy-2-whole2} is finished.
\end{proof}

\subsection{Estimates on $\nabla^2 u$ for Navier-slip boundary condition \eqref{Naiver-slip-condition}.}
We next improve the regularity of the density $\rho$ and velocity $u$ in the bounded domain $\Omega$ with Navier-slip boundary condition \eqref{Naiver-slip-condition}.

\begin{lem}
  Under the condition \eqref{Naiver-slip-condition}, it holds that for $0\leq t\leq T<T^*$,
  \begin{equation}\label{energy-2-domain1}
    \frac{d}{dt}\norm{\sqrt{\rho}u_t}_{L^2}^2+C_0\norm{\nabla u_t}_{L^2}^2\leq C(1+\norm{\sqrt{\rho}\dot{u}}_{L^2})(\norm{\sqrt{\rho}u_t}_{L^2}^2+\norm{\nabla\rho}_{L^2}^2)+C.
  \end{equation}
\end{lem}
\begin{proof}
  Applying $\pa_t$ on $\eqref{CNS-eq}_2$, taking inner product with $u_t$ and integrating by parts yields
  \begin{equation}\label{estimate-2-domain0}
    \begin{aligned}
    &\quad\frac{d}{dt}\int\frac{1}{2}\rho|u_t|^2dx+\int(\mu|\curl u_t|^2+(2\mu+\lambda)(\div u_t)^2)dx\\
    &=\int P_t\div u_tdx-\int\rho u_t\cdot\nabla u\cdot u_tdx-\int\rho u\cdot\nabla(|u_t|^2+u\cdot\nabla u\cdot u_t)dx\\
    &\lesssim\int(|u||\nabla\rho|+|\nabla u|)|\nabla u_t|dx+\int(\rho|u_t|^2|\nabla u|+\rho|u||u_t||\nabla u_t|)dx\\
    &\quad+\int(|u||u_t||\nabla u|^2+|u|^2|u_t||\nabla^2u|+|u|^2|\nabla u||\nabla u_t|)dx\\
    &=\sum_{i=1}^{3}B_i.
    \end{aligned}
  \end{equation}
By Gagliardo-Nirenberg's inequality (see \cite{Nirenberg1959}), \eqref{Hodge-decomposition} and \eqref{energy-1}, one gets
\begin{equation}\label{estimate-2-domain1}
  \begin{aligned}
  B_1&\lesssim(\norm{u}_{L^{\infty}}\norm{\nabla\rho}_{L^2}+\norm{\nabla u}_{L^2})\norm{\nabla u_t}_{L^2}\\
  &\leq\epsilon\norm{\nabla u_t}_{L^2}^2+C(\epsilon)(1+\norm{u}_{L^6}\norm{\nabla u}_{L^6}\norm{\nabla\rho}_{L^2}^2)\\
  &\leq\epsilon\norm{\nabla u_t}_{L^2}^2+C(\epsilon)(1+(\norm{P}_{L^6}+\norm{(\nabla G,\nabla\omega)}_{L^2})\norm{\nabla\rho}_{L^2}^2)\\
  &\leq\epsilon\norm{\nabla u_t}_{L^2}^2+C(\epsilon)(1+\norm{\sqrt{\rho}\dot{u}}_{L^2})\norm{\nabla\rho}_{L^2}^2+C(\epsilon).
  \end{aligned}
\end{equation}
Similarly,
\begin{equation}\label{estimate-2-domain2}
  \begin{aligned}
  B_2&\lesssim\norm{\sqrt{\rho}u_t}_{L^2}\norm{u_t}_{L^6}\norm{\nabla u}_{L^3}+\norm{u}_{L^{\infty}}\norm{\sqrt{\rho}u_t}_{L^2}\norm{\nabla u_t}_{L^2}\\
  &\leq\epsilon\norm{\nabla u_t}_{L^2}^2+C(\epsilon)\norm{\nabla u}_{L^2}\norm{\nabla u}_{L^6}\norm{\sqrt{\rho}u_t}_{L^2}^2\\
  &\leq\epsilon\norm{\nabla u_t}_{L^2}^2+C(\epsilon)(1+\norm{\sqrt{\rho}\dot{u}}_{L^2})\norm{\sqrt{\rho}u_t}_{L^2}^2,
  \end{aligned}
\end{equation}
and
\begin{equation}\label{estimate-2-domain3}
  \begin{aligned}
  B_3&\lesssim\norm{u}_{L^6}\norm{u_t}_{L^6}\norm{\nabla u}_{L^2}\norm{\nabla u}_{L^6}+\norm{u}_{L^6}^2\norm{u_t}_{L^6}\norm{\nabla^2u}_{L^2}+\norm{u}_{L^6}^2\norm{\nabla u}_{L^6}\norm{\nabla u_t}_{L^2}\\
  &\leq\epsilon\norm{\nabla u_t}_{L^2}^2+C(\epsilon)\norm{\nabla u}_{L^2}^4\norm{\nabla^2u}_{L^2}^2\\
  &\leq\epsilon\norm{\nabla u_t}_{L^2}^2+C(\epsilon)\norm{\nabla^2u}_{L^2}^2.
  \end{aligned}
\end{equation}
From the estimate \eqref{Lame-estimate} for the Lam\'{e}'s system,
\begin{equation*}
  \begin{aligned}
  \norm{\nabla^2u}_{L^2}&\lesssim\norm{\rho\dot{u}}_{L^2}+\norm{\nabla P}_{L^2}\\
  &\lesssim\norm{\sqrt{\rho}u_t}_{L^2}+\norm{u}_{L^6}\norm{\nabla u}_{L^3}+\norm{\nabla\rho}_{L^2}\\
  &\lesssim\norm{\sqrt{\rho}u_t}_{L^2}+\norm{\nabla u}_{L^2}^{\frac{1}{2}}\norm{\nabla^2 u}_{L^2}^{\frac{1}{2}}+\norm{\nabla\rho}_{L^2}\\
  &\leq\epsilon\norm{\nabla^2 u}_{L^2}+C(\norm{\sqrt{\rho}u_t}_{L^2}+\norm{\nabla\rho}_{L^2})+C(\epsilon)
  \end{aligned}
\end{equation*}
which yields that
\begin{equation}\label{2rd-gradient-u-domain}
  \norm{\nabla^2u}_{L^2}\leq C(\norm{\sqrt{\rho}u_t}_{L^2}+\norm{\nabla\rho}_{L^2}+1).
\end{equation}
Then combining \eqref{estimate-2-domain0}-\eqref{2rd-gradient-u-domain} and \eqref{Hodge-decomposition} gives that for sufficiently small $\epsilon>0$,
\begin{equation*}
  \frac{d}{dt}\norm{\sqrt{\rho}u_t}_{L^2}^2+C_0\norm{\nabla u_t}_{L^2}^2\leq C(1+\norm{\sqrt{\rho}\dot{u}}_{L^2})(\norm{\sqrt{\rho}u_t}_{L^2}^2+\norm{\nabla\rho}_{L^2}^2)+C.
\end{equation*}
\end{proof}

Next we give two lemmas used later to control the density gradient $\nabla\rho$.

\begin{lem}[Beale-Kato-Majda type inequality]\label{lem-BKM-inequality-domain}
 Let $\Omega\subset\mathbb{R}^3$ be a simply connected bounded domain with smooth boundary and $u$ satisfy \eqref{Naiver-slip-condition}. Then for any $q\in(3,\infty)$, there exists a constant $C>0$ such that
 \begin{equation}\label{gradient-u-infty-domain}
   \norm{\nabla u}_{L^{\infty}}\leq C\norm{(\div u, \curl u)}_{L^{\infty}}\ln(e+\norm{\nabla^2u}_{L^q}(1+\norm{(\div u, \curl u)}_{L^{\infty}})^{-1})+C(\norm{\nabla u}_{L^2}+1).
 \end{equation}
\end{lem}
One can refer to \cite{Cai-Li2023} (Lemma 2.7) for the detailed proof of this lemma. Here we make a slight modification as choosing
\begin{equation*}
  \delta=\min\{1,(\norm{\nabla^2u}_{L^q}(1+\norm{(\div u,\curl u)}_{L^{\infty}})^{-1})^{-q/(q-3)}\}
\end{equation*}
 instead of $\delta=\min\{1,(\norm{\nabla^2u}_{L^q}^{-q/(q-3)}\}$ in Lemma 2.7 of \cite{Cai-Li2023}. Now we give a simple sketch of the proof.
 \begin{proof}
   It follows from \cite{Solonnikov1970,Solonnikov1971} that $u$ can be expressed as
   \begin{equation*}
     \begin{aligned}
     u^i&=\int K_{i,j}(x,y)(\mu\Delta_yu^j+(\mu+\lambda)\nabla_y\div_y u^j)dy\\
     &:=\int K_{i,\cdot}(x,y)(\mu\Delta_yu+(\mu+\lambda)\nabla_y\div_y u)dy\\
     &=(2\mu+\lambda)\int K_{i,\cdot}(x,y)\cdot\nabla_y\div_yudy-\mu\int K_{i,\cdot}(x,y)\cdot\nabla_y\times\curl_yudy\\
     &=I_1+I_2,
     \end{aligned}
   \end{equation*}
   where $K=\{K_{i,j}\}$ with $K_{i,j}(x,y)\in C^{\infty}$ for $x\neq y$ is Green matrix of the Lam\'{e}'s system \eqref{Lame-system} and satisfies that for every multi-indexes $\alpha$ and $\beta$, there exists a constant $C_{\alpha,\beta}$ such that for all $x\neq y$ and $i,j$,
   \begin{equation*}
     |\pa_x^{\alpha}\pa_y^{\beta}K_{i,j}(x,y)|\leq C_{\alpha,\beta}|x-y|^{-1-|\alpha|-|\beta|}.
   \end{equation*}
   Let $\delta\in(0,1]$ be a constant to be determined later and introduce a cut-off function satisfying $\eta_{\delta}(x)=1$ for $|x|<\delta$, $\eta_{\delta}(x)=0$ for $|x|>2\delta$ and $|\nabla\eta_{\delta}|<C\delta^{-1}$. Then due to $K_{i,\cdot}(x,y)\cdot n=0$ on $\pa\Omega$,
   \begin{equation*}
     \begin{aligned}
     \nabla I_1&=(2\mu+\lambda)\int\eta_{\delta}(|x-y|)\nabla_xK_{i,\cdot}(x,y)\cdot\nabla_y\div_yudy\\
     &\quad+(2\mu+\lambda)\int\nabla_y\eta_{\delta}(|x-y|)\cdot\nabla_xK_{i,\cdot}(x,y)\div_yudy\\
     &\quad-(2\mu+\lambda)\int(1-\eta_{\delta}(|x-y|))\nabla_x\div_yK_{i,\cdot}(x,y)\div_yudy\\
     &\lesssim\left(\int_0^{\delta}r^{-2\frac{q}{q-1}}r^2dr\right)^{\frac{q-1}{q}}\norm{\nabla^2u}_{L^q}+\norm{\div u}_{L^{\infty}}\\
     &\qquad+\int_{\delta}^1r^{-3}r^2dr\norm{\div u}_{L^{\infty}}+(\int_1^{\infty}r^{-6}r^2dr)^{\frac{1}{2}}\norm{\nabla u}_{L^2}\\
     &\lesssim\delta^{\frac{q-3}{q}}\norm{\nabla^2u}_{L^q}+(1-\ln\delta)\norm{\div u}_{L^{\infty}}+\norm{\nabla u}_{L^2}.
     \end{aligned}
   \end{equation*}
   Similarly,
   \begin{equation*}
     \begin{aligned}
     \nabla I_2&=-\mu\int\eta_{\delta}(|x-y|)\nabla_xK_{i,\cdot}(x,y)\cdot\nabla_y\times\curl_yudy\\
     &\quad+\mu\int\nabla_y\eta_{\delta}(|x-y|)\times\nabla_xK_{i,\cdot}(x,y)\cdot\curl_yudy\\
     &\quad-\mu\int(1-\eta_{\delta}(|x-y|))\nabla_x\nabla_y\times K_{i,\cdot}(x,y)\cdot\curl_yudy\\
     &\lesssim\delta^{\frac{q-3}{q}}\norm{\nabla^2u}_{L^q}+(1-\ln\delta)\norm{\curl u}_{L^{\infty}}+\norm{\nabla u}_{L^2}.
     \end{aligned}
   \end{equation*}
   Therefore,
   \begin{equation*}
     \norm{\nabla u}_{L^{\infty}}\lesssim \delta^{\frac{q-3}{q}}\norm{\nabla^2u}_{L^q}+(1-\ln\delta)\norm{(\div u,\curl u)}_{L^{\infty}}+\norm{\nabla u}_{L^2}.
   \end{equation*}
   Here choosing $\delta=\min\{1,(\norm{\nabla^2u}_{L^q}(1+\norm{(\div u,\curl u)}_{L^{\infty}})^{-1})^{-q/(q-3)}\}$ gives \eqref{gradient-u-infty-domain} and completes the proof of Lemma \ref{lem-BKM-inequality-domain}.

 \end{proof}

Next lemma comes from \cite{Foias1988}, which can be used to deal with the special $\log$-type Gronwall's inequality.
\begin{lem}
  For $\alpha,\,\beta,\,\sigma>0$, it holds that
  \begin{equation}\label{essential-inequality1}
    \beta^2(1+\ln_{+}\sigma)\leq\frac{1}{2}\alpha^2\sigma^2+\beta^2(1+\ln_{+}\frac{\beta}{\alpha}),
  \end{equation}
  where $\ln_{+}a:=\max\{\ln a, 0\}$ for any $a>0$. In particular, this inequality leads to
  \begin{equation}\label{essential-inequality2}
    \beta^2\ln(e+\sigma)\leq C\alpha^2\sigma^2+C\beta^2\ln(e+\frac{\beta}{\alpha}).
  \end{equation}
  Here $C>0$ is a constant independent of $\alpha,\beta,\sigma$.
\end{lem}
\begin{proof}
  We consider the maximization problem for $f(\sigma)=\beta^2(1+\ln\sigma)-\frac{1}{2}\alpha^2\sigma^2$ for $\sigma>0$. Since the maximum occurs at $\sigma=\beta/\alpha$, we have
  \begin{equation*}
    \beta^2(1+\ln\sigma)-\frac{1}{2}\alpha^2\sigma^2\leq \beta^2(1+\ln\frac{\beta}{\alpha})-\frac{1}{2}\beta^2.
  \end{equation*}
   Distinguishing the case $0<\sigma\leq 1$ and $\sigma>1$, we get
   \begin{equation*}
     \beta^2(1+\ln_{+}\sigma)\leq\frac{1}{2}\alpha^2\sigma^2+\beta^2(1+\ln_{+}\frac{\beta}{\alpha}).
   \end{equation*}
\end{proof}

Now we give the estimate of $\nabla\rho$.

\begin{lem}
  Under the condition \eqref{Naiver-slip-condition}, it holds that for any $q\in(3,\min\{6,\tilde{q}\}]$,
  \begin{equation}\label{energy-2-domain2}
    \frac{d}{dt}\norm{\nabla\rho}_{L^q}^2\leq C(\norm{(\div u,\curl u)}_{L^{\infty}}+1)\norm{\nabla\rho}_{L^q}^2\ln(e+\norm{\nabla\rho}_{L^q}^2)+C(\norm{\nabla u_t}_{L^2}^2+\norm{\nabla^2u}_{L^2}^3).
  \end{equation}
\end{lem}
\begin{proof}
From \eqref{gradient-density},
\begin{equation*}
    \frac{d}{dt}\norm{\nabla\rho}_{L^q}^2
    \leq C(\norm{\nabla u}_{L^{\infty}}+\norm{\div u}_{L^{\infty}}+1)\norm{\nabla\rho}_{L^q}^2+C\norm{\nabla^2 u}_{L^q}^2.
\end{equation*}
Applying \eqref{Lame-estimate} to $\eqref{CNS-eq}_2$ gives
\begin{equation}\label{2rd-gradient-u-Lp}
  \begin{aligned}
    \norm{\nabla^2u}_{L^q}&\lesssim\norm{\rho\dot{u}}_{L^q}+\norm{\nabla P}_{L^q}\\
    &\lesssim\norm{\nabla u_t}_{L^2}+\norm{\nabla\rho}_{L^q}+\norm{u}_{L^{\infty}}\norm{\nabla^2u}_{L^2}\\
    &\lesssim\norm{\nabla u_t}_{L^2}+\norm{\nabla\rho}_{L^q}+\norm{\nabla^2u}_{L^2}^{\frac{3}{2}}.
  \end{aligned}
\end{equation}
Then by \eqref{gradient-u-infty-domain} and \eqref{essential-inequality2},
\begin{equation*}
  \begin{aligned}
  \frac{d}{dt}\norm{\nabla\rho}_{L^q}^2&\leq C(\norm{(\div u,\curl u)}_{L^{\infty}}+1)\norm{\nabla\rho}_{L^q}^2\ln(e+\norm{\nabla^2u}_{L^q}(1+\norm{(\div u, \curl u)}_{L^{\infty}})^{-1})\\
  &\quad+C\norm{\nabla^2u}_{L^q}^2\\
  &\leq C(\norm{(\div u,\curl u)}_{L^{\infty}}+1)\norm{\nabla\rho}_{L^q}^2\ln(e+\norm{\nabla\rho}_{L^q})+C\norm{\nabla^2u}_{L^q}^2\\
  &\leq C(\norm{(\div u,\curl u)}_{L^{\infty}}+1)\norm{\nabla\rho}_{L^q}^2\ln(e+\norm{\nabla\rho}_{L^q}^2)+C(\norm{\nabla u_t}_{L^2}^2+\norm{\nabla^2u}_{L^2}^3).
  \end{aligned}
\end{equation*}
\end{proof}

Next we have the estimates on $\nabla^2u$ and $\nabla\rho$.
\begin{lem}\label{lem-energy-2-domain}
  Under the condition \eqref{Naiver-slip-condition}, it holds that for any $q\in(3,\min\{6,\tilde{q}\}]$,
  \begin{equation}\label{energy-2-domain-final}
    \sup_{0\leq t\leq T}(\norm{(\nabla^2u, \sqrt{\rho}u_t)}_{L^2}^2+\norm{\nabla\rho}_{L^q}^2)+\int_0^T\norm{\nabla u_t}_{L^2}^2dt\leq C, \quad 0\leq T<T^*.
  \end{equation}
\end{lem}

\begin{proof}
  Combining \eqref{energy-2-domain1}, \eqref{energy-2-domain2} and \eqref{2rd-gradient-u-domain} gives that for suitably small $\kappa>0$,
  \begin{equation}\label{estimate-2-domain-final-1}
    \begin{aligned}
    &\quad\frac{d}{dt}(\norm{\sqrt{\rho}u_t}_{L^2}^2+\kappa\norm{\nabla\rho}_{L^q}^2)+C_1\norm{\nabla u_t}_{L^2}^2\\
    &\leq C(\norm{\sqrt{\rho}\dot{u}}_{L^2}+\norm{(\div u,\curl u)}_{L^{\infty}}+1)(\norm{\sqrt{\rho}u_t}_{L^2}^2+\kappa\norm{\nabla\rho}_{L^q}^2+1)\ln(e+\norm{\nabla\rho}_{L^q}^2)\\
    &\quad +C(\norm{\sqrt{\rho}u_t}_{L^2}^3+\norm{\nabla\rho}_{L^2}^3+1).
    \end{aligned}
  \end{equation}
  Since by \eqref{2rd-gradient-u-domain}
  \begin{equation*}
    \begin{aligned}
    \norm{\sqrt{\rho}u_t}_{L^2}&\lesssim\norm{\sqrt{\rho}\dot{u}}_{L^2}+\norm{u\cdot\nabla u}_{L^2}\\
    &\lesssim\norm{\sqrt{\rho}\dot{u}}_{L^2}+\norm{u}_{L^{\infty}}\norm{\nabla u}_{L^2}\\
    &\lesssim\norm{\sqrt{\rho}\dot{u}}_{L^2}+\norm{\nabla^2u}_{L^2}^{\frac{1}{2}}\\
    &\leq\epsilon\norm{\sqrt{\rho}u_t}_{L^2}+C(\norm{\sqrt{\rho}\dot{u}}_{L^2}+\norm{\nabla\rho}_{L^2}^{\frac{1}{2}}+1),
    \end{aligned}
  \end{equation*}
  which yields
  \begin{equation*}
    \norm{\sqrt{\rho}u_t}_{L^2}\leq C(\norm{\sqrt{\rho}\dot{u}}_{L^2}+\norm{\nabla\rho}_{L^2}^{\frac{1}{2}}+1),
  \end{equation*}
  we have
  \begin{equation*}
    \begin{aligned}
    \norm{\sqrt{\rho}u_t}_{L^2}^3&\lesssim(\norm{\sqrt{\rho}\dot{u}}_{L^2}+1)\norm{\sqrt{\rho}u_t}_{L^2}^2+\norm{\nabla\rho}_{L^2}^{\frac{1}{2}}\norm{\sqrt{\rho}u_t}_{L^2}^2\\
    &\leq C(\norm{\sqrt{\rho}\dot{u}}_{L^2}+1)\norm{\sqrt{\rho}u_t}_{L^2}^2+\epsilon\norm{\sqrt{\rho}u_t}_{L^2}^3+C(\epsilon)\norm{\nabla\rho}_{L^2}\norm{\sqrt{\rho}u_t}_{L^2}
    \end{aligned}
  \end{equation*}
  which gives that
  \begin{equation}\label{transition-estimate1}
    \norm{\sqrt{\rho}u_t}_{L^2}^3\leq C(\norm{\sqrt{\rho}\dot{u}}_{L^2}+1)\norm{\sqrt{\rho}u_t}_{L^2}^2+C\norm{\nabla\rho}_{L^2}^2.
  \end{equation}
  Note that
  \begin{equation}\label{transition-estimate2}
    \norm{\nabla\rho}_{L^2}^3\leq C\norm{\rho}_{L^{\frac{12q}{5q-6}}}^2\norm{\nabla\rho}_{L^q}\leq C\norm{\nabla\rho}_{L^q}.
  \end{equation}
  Thus, inserting \eqref{transition-estimate1} and \eqref{transition-estimate2} into \eqref{estimate-2-domain-final-1} yields
  \begin{equation}\label{estimate-2-domain-final-2}
    \begin{aligned}
    &\quad\frac{d}{dt}(\norm{\sqrt{\rho}u_t}_{L^2}^2+\kappa\norm{\nabla\rho}_{L^q}^2)+C_1\norm{\nabla u_t}_{L^2}^2\\
    &\leq C(\norm{\sqrt{\rho}\dot{u}}_{L^2}+\norm{(\div u,\curl u)}_{L^{\infty}}+1)(\norm{\sqrt{\rho}u_t}_{L^2}^2+\kappa\norm{\nabla\rho}_{L^q}^2+1)\ln(e+\norm{\nabla\rho}_{L^q}^2).
    \end{aligned}
  \end{equation}
  Then the $\log$-type Gronwall's inequality \eqref{Gronwall-inequality} gives that
  \begin{equation*}
    \sup_{0\leq t\leq T}(\norm{\sqrt{\rho}u_t}_{L^2}^2+\norm{\nabla\rho}_{L^q}^2)+\int_0^T\norm{\nabla u_t}_{L^2}^2dt\leq C.
  \end{equation*}
  Back to \eqref{2rd-gradient-u-domain}, we have
  \begin{equation*}
    \sup_{0\leq t\leq T}\norm{\nabla^2u}_{L^2}\leq C.
  \end{equation*}
  Hence the proof of Lemma \ref{lem-energy-2-domain} is completed.
\end{proof}

\subsection{The extending of the existence time.}The combination of Lemma \ref{lem-energy-2-whole1} and Lemma \ref{lem-energy-2-whole2} or Lemma \ref{lem-energy-2-domain} is enough to extend the strong solutions of $(\rho,u)$ beyond $t\geq T^*$. Indeed, in view of \eqref{energy-2-whole} or \eqref{energy-2-D} and \eqref{energy-2-whole1} or \eqref{energy-2-domain-final}, the functions $(\rho,u)|_{t=T^*}=\lim_{t\rightarrow T^*}(\rho,u)$ satisfy the conditions imposed on the initial data \eqref{initial-data-condition} at the time $t=T^*$. In particular,
\begin{equation*}
  -\mu\Delta u-(\mu+\lambda)\nabla\div u+\nabla P|_{t=T^*}=\lim_{t\rightarrow T^*}(\rho u_t+\rho u\cdot\nabla u)=\sqrt{\rho}g|_{t=T^*},
\end{equation*}
with $g|_{t=T^*}\in L^2$. Thus, $(\rho,u)_{t=T^*}$ also satisfy \eqref{compatibility-condition}. Hence, we can take $(\rho,u)_{t=T^*}$ as the initial data and apply the local existence theorem \ref{thm-local-existence} to extend the local strong solution beyond $T^*$. This contradicts the assumption on $T^*$.

\medskip

\section*{\bf Data availability}
No data was used for the research described in the article.

\section*{\bf Conflicts of interest}

The authors declare no conflict of interest.

\end{document}